\documentclass[a4paper,11pt]{amsart}
\usepackage{etex}
\usepackage{tikz, tikz-cd}

\usepackage{bm}
\usepackage{amssymb}
\usepackage{mathrsfs}
\usepackage{amsmath, amssymb,amsthm,latexsym,amscd,mathrsfs}
\usepackage{indentfirst}
\usepackage{stmaryrd}
\usepackage{graphicx}
\usepackage{subfigure}
\usepackage{extarrows}
\usepackage{amssymb}
\usepackage{amsmath,amssymb,amsthm,latexsym,amscd,mathrsfs}
\usepackage{indentfirst}
\usepackage{multirow}
\usepackage{latexsym}
\usepackage{amsfonts}
\usepackage{color}
\usepackage{pictexwd,dcpic}
\usepackage{graphicx}
\usepackage{psfrag}
\usepackage{hyperref}
\usepackage{comment}
\usepackage{cases}
\usepackage{dutchcal}

\makeatletter
\@namedef{subjclassname@2010}{{\rm2020} Mathematics Subject Classification}
\makeatother

 \newcommand{\R}{\mathbb{R}}
 \newcommand{\C}{\mathbb{C}}
 \newcommand{\ROM}[1]{\mathrm{\uppercase\expandafter{\romannumeral#1}}}
  \theoremstyle{definition}
   \numberwithin{equation}{section}
   \theoremstyle{plain}

 \newtheorem{thm}{Theorem}

 \newtheorem{cor}{Corollary}
  
 \newtheorem{rem}{Remark}
 \newtheorem{prop}{Proposition}

\newtheorem{ack}{Acknowledgements}   
  \makeatletter

  \allowdisplaybreaks
\makeatother
\title[Isoparametric hypersurfaces and complex structures
]{\textbf{Isoparametric hypersurfaces and complex structures}}
\author[Z. Z. Tang]{Zizhou Tang}\address{Chern Institute of Mathematics $\&$ LPMC, Nankai University, Tianjin 300071, P. R. China}
\email{zztang@nankai.edu.cn}
\author[W. J. Yan]{Wenjiao Yan}\address{School of Mathematical Sciences, Laboratory of Mathematics and Complex Systems, Beijing Normal University, Beijing, 100875, P. R. China}
\email{wjyan@bnu.edu.cn}
\thanks {The project is partially supported by the NSFC (No.11871282, 11931007), BNSF (Z190003), and Nankai Zhide Foundation. 
}


\begin{document}

\maketitle
\begin{center}
Dedicated to Professor Banghe Li on the occasion of his 80th birthday
\end{center}
\begin{abstract}
The main purpose of this note is to construct almost complex or complex structures on certain isoparametric hypersurfaces in unit spheres. As a consequence, complex
structures on $S^1\times S^7\times S^6$, and on
$S^1\times S^3\times S^2$ with vanishing first Chern class, are built.
\end{abstract}


An almost complex manifold is a smooth manifold $M^{2n}$ that admits an almost complex structure $J$, namely, an endomorphism $J$ of the tangent bundle $TM$ with $J^2=-I$. On such an $M^{2n}$, one can locally choose a field of $U$-bases $e_1,\cdots, e_{2n}$ satisfying
\begin{equation*}\label{J0}
J(e_1,e_2,...,e_{2n}) = (e_1,e_2,...,e_{2n})J_0,
\end{equation*}
where $J_0 =\begin{pmatrix}0&-I_n\\I_n&0\end{pmatrix}$. 
The existence problem for an almost complex structure on a given oriented manifold $M^{2n}$ is a purely topological one, and is understood
relatively well. For example, it is well known that the only spheres that admit
almost complex structures are $S^2$ and $S^6$.

A complex manifold admits a natural almost complex structure induced by the holomorphic coordinates, i.e., the one that is called a complex structure. If an almost complex structure $J$ is induced from a complex structure, it is said to be integrable. There are various examples of complex manifolds: $2$-dimensional orientable surfaces, the Hopf manifold with underlying spaces $S^1\times S^{2n-1}$ (see for example \cite{Che79}, \cite{PT05}), the Calabi-Eckmann manifold with underlying spaces $S^{2p+1}\times S^{2q+1}$ $(p,q\geq 1)$ (cf. \cite{CE53}), and the manifold $S^1\times \Sigma^{2n-1} $, where $\Sigma^{2n-1}$ is a homotopy sphere bounding a parallelizable manifold (cf. \cite{BV68}), etc..
However, the very natural almost complex structure on $S^6$ induced by octonionic multiplication is not integrable. It is challenging and still an open question as to whether or not $S^6$ admits a complex structure. One is also interested in the existence or non-existence of complex structures on $S^2\times S^6$, $S^1\times S^1\times S^6$, $S^6\times S^6$, even on $P \times S^6 $, the product of any even dimensional closed manifold $P$ with $S^6$.

In this paper, we construct almost complex structures on certain isoparametric hypersurfaces (those with constant principal curvatures) in unit spheres. The classification of isoparametric hypersurfaces  in unit spheres was accomplished recently (cf. \cite{CCJ07}, \cite{Imm08},  \cite{Miy13}, \cite{Chi13}, \cite{Chi20}).
Among of these, the ones with $4$ distinct principal curvatures, especially of the OT-FKM type, carry the most abundant topology and geometry. 
For a symmetric Clifford system $\{P_0,\cdots, P_m\}$ on $\mathbb{R}^{2l}$, i.e., $P_{i}$'s are symmetric matrices satisfying $P_{i}P_{j}+P_{j}P_{i}=2\delta_{ij}I_{2l}$, the following homogeneous polynomial $F$ of degree $4$ on
$\mathbb{R}^{2l}$ is constructed (cf.  \cite{FKM81}):
\begin{equation*}\label{FKM isop. poly.}
\begin{split}
&\qquad F:\quad \mathbb{R}^{2l}\longrightarrow \mathbb{R}\\
&F(x) = |x|^4 - 2\sum_{i = 0}^{m}{\langle
P_{i}x,x\rangle^2}.
\end{split}
\end{equation*}
The restricted function $f=F|_{S^{2l-1}}$ has $\mathrm{Im}f=[-1,1]$. 
A regular level set $M^{2l-2}$ of $f$ is called \emph{an isoparametric hypersurface of OT-FKM type}. The singular level sets $M_+=f^{-1}(1)$, $M_-=f^{-1}(-1)$ are called \emph{focal submanifolds}. The multiplicities of the principal curvatures of $M^{2l-2}$ are $(m_1, m_2, m_1, m_2)$, where $(m_1, m_2)=(m, l-m-1)$ with
$l=k\delta(m)$, $k$ is a positive integer and $\delta(m)$ is the dimension of the irreducible module of the Clifford algebra $\mathcal{C}_{m-1}$, which is valued as
\begin{center}
\begin{tabular}{|c|c|c|c|c|c|c|c|c|cc|}
\hline
$m$ & 1 & 2 & 3 & 4 & 5 & 6 & 7 & 8 & $\cdots$ &$m$+8 \\
\hline
$\delta(m)$ & 1 & 2 & 4 & 4 & 8 & 8 & 8 & 8 &$\cdots$ &16$\delta(m)$\\
\hline
\end{tabular} .
\end{center}
The isoparametric hypersurface
$M^{2l-2}$ is an $S^{m_1}$ $(S^{m_2})$ bundle over $M_+$ $(M_-)$. In particular, the normal bundle of $M_+$ in $S^{2l-1}$ is trivial with $(P_0x,\cdots, P_mx)$ as a normal field. Hence, $M^{2l-2}$ is diffeomorphic to $S^m\times M_+$. 

First, we construct an almost complex structure $J$ on the isoparametric hypersurface $M_0=F^{-1}(0)\cap S^{2l-1}$.  With respect to the induced metric $ds^2$, we show the following:

\begin{prop}\label{M0}
 $(M_0, J, ds^2)$ is an almost Hermitian manifold.
\end{prop}

\begin{proof}
$M_0$ is a regular level set of $f$, thus a unit normal vector at $x$ can be chosen as
$$\xi(x)=\frac{\nabla f}{|\nabla f|}=\frac{DF-\langle DF, x\rangle x}{|DF-\langle DF, x\rangle x|}=\frac{DF}{|DF|}=x-2\sum_{i = 0}^{m}\langle
P_{i}x,x\rangle P_ix,$$
where $\nabla $ is the spherical gradient and $D$ the Euclidean gradient.
 Let
$$\eta_1:=P_0P_1x, \quad \eta_2:=P_0P_1\xi.$$
Using the properties of Clifford matrices, we can verify that $\langle \eta_i, x\rangle=\langle \eta_i, \xi\rangle=0$ for $i=1,2$. Thus $\eta_1, \eta_2\in T_xM_0$ and the tangent space $T_xM_0$ is split into
$$T_xM_0=\textrm{Span}\{\eta_1\}\oplus \textrm{Span}\{\eta_2\}\oplus E.$$
For $X\in E$, it is easy to check directly that 
$P_0P_1X\in E$. This helps us to construct a global endomorphism of $TM_0$ as follows: 
\begin{equation}\label{F0}
\begin{split}
&\hspace{1cm}J:\quad T_xM_0\longrightarrow T_xM_0\\
& J\eta_1:=\eta_2, ~~J\eta_2:=-\eta_1, ~~JX:=P_0P_1X ~\textrm{for}~X\in E.
\end{split}
\end{equation}
It is obvious that $J^2=-I$, and is thus an almost complex structure on $M_0$. Clearly, $J$ is compatible with the induced metric $ds^2$ on $M_0$, and is thus almost Hermitian.

\end{proof}

\begin{rem}
The almost complex structure $J$ in (\ref{F0}) is a very natural one and its integrability will be investigated later. In the case $m=1$,  we will define a slightly different almost complex structure on $M$ and prove its integrability.
\end{rem}

\begin{rem}
One can also consider similar problems on other isoparametric hypersurfaces in unit spheres. For instance, Miyaoka \cite{Miy11} showed that the isoparametric hypersurface $M^{12}$ in $S^{13}$ with $6$ distinct principal curvatures has a K\"ahler structure with respect to the induced metric, and the Cartan isoparametric hypersurface $C_{\C}^6\cong SU(3)/T^2$ is a K\"ahler submanifold of $M^{12}$.
\end{rem}

From now on, we focus on the case $m=1$. Consider $N:=S^1\times M_+$ with the product metric; it is diffeomorphic to (but not isometric to) the isoparametric hypersurface $M$. Define
\begin{equation*}\label{e}
\varepsilon(a):=\sqrt{-1}a\in T_aS^1, ~~ \textrm{for}~ a\in S^1, \qquad\,\, ~~\eta(x):=P_0P_1x,~~ \textrm{for}~x\in M_+.
\end{equation*}
Clearly, $\eta(x)\in T_xM_+$. We split $T_xM_+$ as $T_xM_+=\mathrm{Span}\{\eta\}\oplus V$. For $X\in V$, we have
$\langle P_0P_1X, x\rangle=\langle P_0P_1X, P_0x\rangle=\langle P_0P_1X, P_1x\rangle=\langle P_0P_1X, \eta\rangle=0$, which implies $P_0P_1X\in V$. Based on these facts, we define an endomorphism $J$ of $TN$ as follows:
\begin{equation}\label{J}
\begin{split}
&\hspace{2cm}J:~~\, T_{(a, x)}N\longrightarrow T_{(a, x)}N\\
&J\varepsilon:=\eta, ~~\,J\eta:=-\varepsilon,\,\,\,\textrm{and}~JX:=P_0P_1X, ~\textrm{for}~ X\in V.
\end{split}
\end{equation}

\begin{thm}\label{m1}
When $m=1$, the almost complex structure $J$ on $S^1\times M_+$ defined in (\ref{J}) is integrable. Furthermore, $(S^1\times M_+, J, ds^2)$ is a Hermitian manifold, where $ds^2$ is the product metric on $S^1\times M_+$.
\end{thm}

\begin{proof}
We first recall that there always exists a Clifford system $\{P_0, P_1\}$ on $\R^{2n}$ as follows:
\begin{equation}\label{P0P1}
P_0=\begin{pmatrix} I_n&0\\0&-I_n\end{pmatrix},\quad P_1=\begin{pmatrix} 0&I_n\\I_n&0\end{pmatrix}.
\end{equation}
Therefore, there exists the following integrable almost complex structure $J_{\R}$ on $\R^{2n}$:
\begin{eqnarray*}
J_{\R}:\quad T_x\R^{2n}&\longrightarrow& T_x\R^{2n}\\
Y&\mapsto& P_0P_1Y.
\end{eqnarray*}


In the case $m=1$, $l=k\delta(m)=k$, we will show that $J$ in (\ref{J}) is just the almost complex structure induced from $\R^{2k}$, and is thus integrable.
Locally, we consider the embedding
\begin{eqnarray*}
\Phi:~~(S^1\setminus \{1\})\times M_+&\longrightarrow & \R^{2k}\\
(a, ~x)&\mapsto& a^{-\sqrt{-1}}x .
\end{eqnarray*}
One can check directly that 
$$J_{\R}\Phi_*(\varepsilon, 0)=\Phi_*(0, \eta)=\Phi_*(J(\varepsilon, 0))$$
 and
$$J_{\R}\Phi_*(0, X)=P_0P_1\Phi_*(0, X)=\Phi_*(0, P_0P_1X)= \Phi_*(J(0, X)) \;\;for\; X\in V.$$
Therefore, $J=\Phi^*J_{\R}$. 
Moreover, taking the metric on $M_+$ that is induced from the unit sphere, i.e., from the Euclidean space, it is obvious 
that $J$ is compatible with the product metric $ds^2$ on $S^1\times M_+$.
\end{proof}

According to \cite{FKM81}, $M_+$ of OT-FKM type is a sphere bundle over a sphere. Based on \cite{Wan88}, \cite{QTY21} determined when the sphere bundle is trivial. In particular, among the OT-FKM family with $m=1$, there are only two trivial cases: 
$M_+^{5}\cong S^3\times S^2$ in the case where $(m_1, m_2)=(1,2)$, and $M_+^{13}\cong S^7\times S^6$ in the case where $(m_1, m_2)=(1,6)$. Furthermore, the corresponding isoparametric hypersurfaces have diffeomorphisms 
 $M^{6}\cong S^1\times S^3\times S^2$ and $M^{14}\cong S^1\times S^7\times S^6$, respectively. As a corollary, we obtain

\begin{cor}\label{N14}
$S^1\times S^7\times S^6$ is a complex manifold.
\end{cor}

\begin{rem}\label{metric}
When $m=1$, $M_+$ with the induced metric doesn't have non-negative sectional curvatures (\cite{QTY21}). Therefore,
$M_+$ with $(m_1, m_2)=(1,2)$ or $(1, 6)$ is not isometric to products of round spheres $S^3\times S^2$ or $S^7\times S^6$, no matter what the radii of the spheres are.
\end{rem}

Let $J_1$ be a complex structure (e.g., the Hopf structure) on $S^1\times S^3$, and $J_2$ be the canonical complex structure on $S^2$ whose first Chern class is twice a generator of the second cohomology. Then $(J_1, J_2)$ is a complex structure on $S^1\times S^3\times S^2$ whose first Chern class is not zero (cf. \cite{Uen75}). 
We 
will show that our complex structure $J$ in (\ref{J}) on $S^1\times M_+^5$, where $M_+^5$ (diffeomorphic to $S^3\times S^2$) is the focal submanifold of OT-FKM type  with $(m_1, m_2)=(1,2)$, is quite different from the product complex structure $(J_1, J_2)$.

\begin{prop}\label{12}
The Hermitian manifold $S^1\times M_+^5$ admits an orthonormal $U$-frame.
Thus, the first Chern class of $J$ in (\ref{J}) vanishes.
\end{prop}

\begin{proof}
We will start with giving explicitly the almost complex structures $\widetilde{J}$ on $\widetilde{N}^6:=S^1\times S^3\times S^2$ and $\widetilde{N}^{14}:=S^1\times S^7\times S^6$, which are induced from $J$ on $N^6:=S^1\times M_+^5$ and $N^{14}:=S^1\times M_+^{14}$, respectively.
Here, we regard $S^3$, $S^7$ as unit spheres in $\mathbb{H}\cong\R^4$ and $\mathbb{O}\cong\R^8$, and $S^2$, $S^6$ as unit spheres in $\R^3:=\mathrm{Im} \mathbb{H}$ and $\R^7:=\mathrm{Im}\mathbb{O}$, respectively.

In the case $m=1$, we always take the Clifford system $\{P_0, P_1\}$ in (\ref{P0P1}), so the focal submanifold $M_+$ can be  expressed as
\begin{equation*}\label{M+}
M_+
= \left\{(u, v) \in \mathbb{R}^l\oplus\mathbb{R}^l~\vline~\begin{array}{ll}\langle u, v \rangle=0, \,\,|u|^2=|v|^2=\frac{1}{2}\end{array}\right\},
\end{equation*}
which is a Stiefel manifold. Constructing the 
diffeomorphism
\begin{eqnarray*}
\psi:~~S^7\times S^6&\longrightarrow& M_+\\
(y, ~~z)&\mapsto& \frac{1}{\sqrt{2}}(y, ~~zy),
\end{eqnarray*}
we derive the diffeomorphism
\begin{eqnarray*}
\Psi=\mathrm{id}\times\psi:~~\widetilde{N}&\longrightarrow& S^1\times M_+\\
(a,~~y, ~~z)&\mapsto& \big(a, ~~\frac{1}{\sqrt{2}}(y, ~~zy)\big).
\end{eqnarray*}
At $(y, z)\in S^7\times S^6$, for $(Y, Z)\in T_yS^7\times T_zS^6$, the tangent map $\psi_*$ is
$$\psi_*(Y, Z)=\frac{1}{\sqrt{2}}(Y, ~~Zy+zY).$$
Clearly, at $x:=\psi(y, z)\in M_+$, we have that $\eta(\psi(y,z))=\frac{1}{\sqrt{2}}(zy, -y)=\psi_*(zy, 0)$.
For simplicity, we will identify $\eta(\psi(y,z))$ with $(zy, 0)$, and still denote that $\eta:=(zy, 0)$.

For $(Y, Z)\in T_yS^7\times T_zS^6$ with $(Y, Z)\perp \eta$, i.e., $\langle Y, zy \rangle=0$, it follows from the Artin Theorem that $\langle zY, -y\rangle=-\langle \bar{z}y, Y\rangle=\langle zy, Y\rangle=0$, and thus that
$$\langle \psi_*(Y, Z), \eta\rangle=\langle\frac{1}{\sqrt{2}}(Y, Zy+zY), \frac{1}{\sqrt{2}}(zy, -y)\rangle=\frac{1}{2}\langle Y, zy\rangle+\frac{1}{2}\langle zY, -y\rangle=0.$$
We write explicitly the almost complex structure $\widetilde{J}$ on  $\widetilde{N}^{14}$ as
\begin{align}
\widetilde{J}(\varepsilon, ~~0, ~~0)&=(0,~~zy,~~0)\nonumber\\
\widetilde{J}(0,~~zy,~~0)&=(-\varepsilon, ~~0, ~~0)\label{JS}\\
\widetilde{J}(0, ~~Y, ~~Z)&=\Big(0, ~Zy+zY, ~-\big(z(Zy)\big)\bar{y}\Big),\,\,~~\textrm{if}~ \langle Y, zy\rangle=0.\nonumber
\end{align}
It is obvious that $\widetilde{J}$ is not Hermitian with respect to the product metric on $\widetilde{N}^{14}$.

In particular, since the multiplication on $\mathbb{H}$ satisfies the associativity,  $\widetilde{J}$ on $\widetilde{N}^6=S^1\times S^3\times S^2$ can be simplified to
\begin{align}
\widetilde{J}(\varepsilon, ~~0, ~~0)&=(0,~~zy,~~0)\nonumber\\
\widetilde{J}(0,~~zy,~~0)&=(-\varepsilon, ~~0, ~~0)\label{JS'}\\
\widetilde{J}(0, ~~Y, ~~Z)&=\Big(0, ~Zy+zY, ~-zZ\Big),\,\,~~\textrm{if}~ \langle Y, zy\rangle=0.\nonumber
\end{align}

Next, we will use (\ref{JS'}) to prove Proposition \ref{12}.

Let $\{1, e_1, e_2, e_3\}$ be the standard orthonormal basis of $\mathbb{H}$. We regard $S^2$ as $S^2\subset \R^3:=\mathrm{Im} \mathbb{H}=\mathrm{Span}\{e_1, e_2, e_3\}$; i.e., for any $z\in S^2$, $z=z_1e_2+z_2e_2+z_3e_3:=(z_1, z_2, z_3)$. Following \cite{PT05}, we define $6$ vector fields on $\widetilde{N}^6$ as follows: at $(a, y, z)\in \widetilde{N}^6$,
$$V_j:= \big( \langle e_j, z\rangle\sqrt{-1}a, ~0, ~e_j-\langle e_j, z\rangle z \big), \quad U_j:= ( 0, ~e_j y, ~0 ), \quad j=1,2,3.$$
One can check easily that $\{V_1, V_2, V_3, U_1, U_2, U_3\}$ is a (global) field of orthonormal frames on $\widetilde{N}^6$. Let $S$ be the representation matrix of $\widetilde{J}$ under $\{V_1, V_2, V_3, U_1, U_2, U_3\}$ 
\begin{equation*}\label{S}
\widetilde{J}(V_1, V_2, V_3, U_1, U_2, U_3)=(V_1, V_2, V_3, U_1, U_2, U_3)S.
\end{equation*}
Using the expression of $\widetilde{J}$ in (\ref{JS'}), we obtain that
$$\widetilde{J}V_j=\big( 0, ~e_jy, ~-ze_j-\langle e_j, z\rangle \big),\quad \widetilde{J}U_j=\big( -\langle e_j, z\rangle\sqrt{-1}a, ~ ze_jy+\langle e_j, z\rangle y, ~0 \big).
$$
Therefore, for $i, j=1, 2, 3$,
\begin{align*}
\langle \widetilde{J}V_j, V_i\rangle=\langle e_ie_j, z\rangle,\,\,\,\,~~ &\quad \langle \widetilde{J}V_j, U_i\rangle=\delta_{ij}\\
\langle \widetilde{J}U_j, U_i\rangle=-\langle e_ie_j, z\rangle, &\quad \langle \widetilde{J}U_j, V_i\rangle=-\langle e_i, z\rangle\langle e_j, z\rangle.
\end{align*}
Denoting that
$h(z):=\begin{pmatrix}0&z_3&-z_2\\-z_3&0&z_1\\z_2&-z_1&0\end{pmatrix}$, we can express the matrix $S$ as
\begin{equation}\label{S6}
S=\begin{pmatrix}h&-z^tz\\I&-h\end{pmatrix}.
\end{equation}
Moreover, it holds that $h^t=-h$, $hz^t=0$ and $h^2+I_3=z^tz$, which implies that $S^2=-I_6$.

Next, we will look for a $U$-frame on $S^1\times M_+^5$ by virtue of $S$.
Let $(X_1,\cdots, X_6)$ be a frame on $\widetilde{N}^6$ such that
$$\widetilde{J}(X_1,\cdots, X_6)=(X_1,\cdots, X_6)J_0,$$
where
$J_0=\begin{pmatrix}0&-I_3\\I_3&0\end{pmatrix}$.
Suppose that $(X_1,\cdots, X_6)=(V_1,V_2, V_3, U_1, U_2, U_3)A$.
Then we have that $$A^{-1}SA=J_0.$$
With $S$ expressed above, we find that $A=\begin{pmatrix}I_3&h\\0&I_3\end{pmatrix}$, with $A^{-1}=\begin{pmatrix}I_3&-h\\0&I_3\end{pmatrix}$, satisfies the equality above.
For convenience, we denote that $V:=(V_1, V_2, V_3)$ and $U:=(U_1, U_2, U_3)$, so
$$(X_1,\cdots,X_6)=(V_1,V_2, V_3, U_1, U_2, U_3)A=(V, ~~Vh+U).$$
Since
\begin{align*}
\psi_*V_j&=\big(\langle e_j, z\rangle\sqrt{-1}a, ~0, ~\frac{1}{\sqrt{2}}(e_j-\langle e_j, z\rangle z)y\big),\\
\psi_*U_j&=\big(0, ~\frac{1}{\sqrt{2}}e_jy, ~\frac{1}{\sqrt{2}}ze_jy\big),
\end{align*}
we can define a frame $\{X_1^*, \cdots, X_6^*\}$ on $S^1\times M_+^5$ as follows:
\begin{align*}
X_{j}^*&:=\psi_*(X_{j})=\big(\langle e_j, z\rangle\sqrt{-1}a, ~~0, ~~\frac{1}{\sqrt{2}}(e_jy-\langle e_j, z\rangle zy)\big), ~~\,j=1,2,3,\\
X_{j+3}^*&:=\psi_*(X_{j+3})=\psi_*(\sum_{k=1}^3V_kh_{kj}+U_j)=(0, ~~\frac{1}{\sqrt{2}}e_jy, ~~-\frac{1}{\sqrt{2}}z_jy), ~~\,j=1,2,3.
\end{align*}
A direct calculation leads to
$$\langle X_i^*, X_j^*\rangle=\langle X_{i+3}^*, X_{j+3}^*\rangle=\frac{1}{2}\delta_{ij}+\frac{1}{2}z_iz_j, ~~\,\, \langle X_i^*, X_{j+3}^*\rangle=0,~~\,\,1\leq i, j\leq 3.$$
Now we modify $X_{j}^*$ and set that 
\begin{align*}
Y_j&:=\Big(\langle e_j, z\rangle\sqrt{-1}a, ~0, ~\big(e_j-\langle e_j, z\rangle z\big)y\Big), \,\,\,\,1\leq j\leq 3,\\
(Y_1,\cdots, Y_6)&:=(Y_1, Y_2, Y_3, JY_1, JY_2, JY_3).
\end{align*}
It is clear that $\langle Y_A, Y_B \rangle =\delta_{AB}$ for $1\leq A, B\leq 6$. Therefore, $\{Y_1,\cdots, Y_6\}$ is the $U$-frame that we demanded.

Moreover, denote by $\omega_{AB}$ $(1\leq A,B\leq 6)$ the connection form corresponding to $Y_1,\cdots, Y_6$. According to \cite{Tan06}, on the Hermitian manifold $S^1\times M_+^5$, $\sum_{j}d\omega_{j, j+3}$ is a  closed $2$-form independent of the choice of frame. Furthermore, the first Chern class can be expressed by
$$c_1(J)=-\frac{1}{2\pi}\sum_{j=1}^3d\omega_{j, j+3}=d(-\frac{1}{2\pi}\sum_{j=1}^3\omega_{j, j+3}),$$
which is globally exact. Hence, the first Chern class of this $J$ vanishes.

\end{proof}
Recall that an almost Hermitian manifold is called \emph{almost K\"ahler} if the K\"ahler $2$-form $\Phi(X, Y):=\langle X, JY\rangle$ is closed, where $X, Y$ are smooth vector fields, and called \emph{nearly K\"ahler} if the $3$-covariant tensor $\nabla\Phi$ is totally antisymmetric. 
As the last result of this paper, we have

\begin{prop}\label{16}
On the complex manifold $S^1\times M_+^{13}$, where $M_+^{13}$ is the focal submanifold of OT-FKM type  with $(m_1, m_2)=(1, 6)$, the Hermitian structure is neither almost K\"ahler, nor nearly K\"ahler.
\end{prop}

\begin{proof}
The diffeomorphism $S^1\times M_+^{13}\cong S^1\times S^7\times S^6$ gives $H^2_{DR}(S^1\times M_+^{13}, \R)= 0$, and thus $S^1\times M_+^{13}$ is not almost K\"ahler.

The covariant derivative of the K\"ahler form $\Phi(X, Y):=\langle X, JY\rangle$ is
\begin{align*}
\nabla\Phi(X, Y, Z)&=(\nabla_Z\Phi)(X, Y)=Z\Phi(X, Y)-\Phi(\nabla_ZX, Y)-\Phi(X, \nabla_ZY)\\
&=\langle X, \nabla_Z(JY)-J\nabla_ZY\rangle.
\end{align*}
On $S^1\times M_+^{13}$, taking $X\in V\subset T_xM_+$ with $|X|=1$, it is apparent that
$$\nabla\Phi(X, \eta, X)= \langle X,  \nabla_X(J\eta)-J\nabla_X\eta\rangle\\
= \langle X, -\nabla_X\varepsilon-J\nabla_X\eta\rangle=\langle JX, \nabla_X\eta\rangle.$$
Noticing that
$\nabla_X\eta=\nabla_XP_0P_1x=(D_XP_0P_1x)^T=(P_0P_1X)^T,$
where $D$ is the Euclidean connection and $(\cdot)^T$ denotes the tangential projection to $M_+$, we derive that
$$\nabla\Phi(X, \eta, X)=\langle JX, (P_0P_1X)^T\rangle=\langle JX, P_0P_1X\rangle=|X|^2=1,$$
which implies that $\nabla\Phi$ is not totally antisymmetric, i.e., the Hermitian manifold $S^1\times M_+^{13}$ is not nearly K\"ahler.

\end{proof}
\begin{rem}
As a matter of fact, one can find from the proof above that all of the Hermitian manifolds $S^1\times M_+$ are not nearly K\"ahler.
\end{rem}

\begin{ack}
The second author would like to thank Professor Jianquan Ge and Doctor Yi Zhou for helpful discussions.
\end{ack}

\end{document}